\def \version {22nd November, 2013}

\documentclass[a4paper,12pt]{article}
\usepackage{array,amsfonts,amsmath,graphicx,mathrsfs,amssymb,amsthm,latexsym}
\usepackage[latin1]{inputenc}
\newtheorem{thm}{Theorem}[section]
\newtheorem{lemma}[thm]{Lemma}

\def \cH {{\cal H}}

\def \cP {{\cal P}}

\begin{document}
\title{Uniformly resolvable decompositions of $K_v$\\ into $P_3$ and $K_3$ graphs}

\author { Salvatore Milici
\thanks{Supported by MIUR and by C. N. R. (G. N. S. A. G. A.), Italy}\\
\small Dipartimento di Matematica e Informatica \\
\small Universit\`a di Catania \\
\small Catania\\
\small Italia\\
{\small \tt milici@dmi.unict.it} \and Zsolt Tuza
\thanks{also affiliated with the
 Department of Computer Science and Systems Technology,
 University of Pannonia, Veszpr\'em, Hungary. }
\thanks{Supported in part
   by the Hungarian Scientific Research Fund, OTKA grant T-81493,
  and by the Hungarian State and the European Union under the grant
TAMOP-4.2.2.A-11/1/KONV-2012-0072. }\\
\small Alfr\'ed R\'enyi Institute of Mathematics \\
\small Hungarian Academy of Sciences \\
\small Budapest \\
\small Hungary \\
{\small \tt tuza@dcs.uni-pannon.hu} \and }

\date{\small Latest update on \version }
\maketitle

\begin{abstract}
In this paper we consider the uniformly resolvable decompositions
of the complete graph $K_v$, or the complete graph minus a
1-factor as appropriate, into subgraphs such that each
resolution class contains only blocks isomorphic to the same
graph. We completely determine the spectrum for the
 case in which all the resolution
classes are either $P_3$ or $K_3$.
\end{abstract}

\vspace{5 mm}

\section{Introduction and Definitions}\label{introduzione}

Given a collection $\cH$ of graphs, an {\em $\cH$-decomposition\/}
of a graph $G$ is a decomposition of the edge set of $G$ into
subgraphs isomorphic to the
 members of $\cH$. The copies of $H\in\cH$ in
the decomposition are called {\em blocks}. Such a decomposition is
called {\em resolvable\/} if it is possible to partition the blocks
into {\em classes\/} $\cP_i$
 (often referred to as {\em parallel classes\/})
 such that every vertex of $G$ appears
in exactly one block of each $\cP_i$.

A resolvable $\cH$-decomposition of $G$ is sometimes also
referred to as an {\em $\cH$-fac\-torization\/} of $G$, and a class can
be called an {\em $\cH$-factor\/} of $G$. The case where
 $\cH=K_2$ (a single edge) is known as a
 {\em 1-factorization\/}; for $G=K_v$
it is well known to exist if and only if $v$ is
even. A single class of a  1-factorization, that is a pairing of
all vertices, is also known as a {\em 1-factor\/}
 or {\em perfect matching}.

In many cases we
wish to place further constraints on the classes. For example, a
class is called {\em uniform} if every block of the class is
isomorphic to the same graph from $\cH$. Of particular note is the
result of Rees \cite{R} which finds necessary and sufficient
conditions for the existence of uniform $\{K_2,
K_3\}$-decompositions of $K_v$. Uniformly resolvable
decompositions of $K_v$ have also been studied in \cite{DQS},
\cite{DLD}, \cite{GM}, \cite{HR}, \cite{KMT}, \cite{M}, \cite{S1} and \cite{S2}.

In this paper we study the existence of uniformly resolvable
decompositions into paths $P_3$ and cycles $K_3\cong C_3$
 (both having three vertices)
for the complete graph $K_v$ and for the complete graph minus a
1-factor, which we denote by $K_v-I$.
The existence of resolvable decompositions for each of $P_3$
and $K_3$ was studied separately already long ago:
\begin{itemize}
 \item
  There exists a resolvable $K_3$-decomposition of $K_v$
  (called {\em Kirkman Triple System}, denoted as KTS$(v)$)
 if and only if \ $v\equiv 3\pmod {6}$.
 \item
  There exists a resolvable $K_3$-decomposition of $K_v-I$
  (called {\em Nearly Kirkman Triple System}, denoted as NKTS$(v)$)
 if and only if \ $v\equiv 0\pmod {6}$ \ and $v > 12$ \ \cite{RS}.
 \item
  There exists a resolvable $P_3$-decomposition of $K_v$
 if and only if \ $v\equiv 9\pmod {12}$ \ \cite{H}.
 \item
  There exists a resolvable $P_3$-decomposition of $K_v-I$
 if and only if \ $v\equiv 6\pmod {12}$.
  (This follows from the case $v=6$ and from the spectrum
  of KTS$(v)$ systems.)
\end{itemize}
Further results on resolvable path decompositions are given in \cite{HR72}.

Let now
\begin{itemize}
 \item $G=K_v$ for $v$ odd,
 \item $G=K_v-I$ for $v$ even,
\end{itemize}
 and let
\begin{quote}
  $URD(v;P_3,K_3)$ := $\{(r,s)$ :
  there exists a uniformly resolvable decomposition of
$G$ into $r$ classes containing only copies of $P_3$ and $s$
classes containing only copies of $K_3\}$.
\end{quote}
%
For $v\geq3$, divisible by 3,
 define $I(v)$ according to the following table, where the first two
 lines are meant for $v\ge 18$ only:

\vspace{4 mm}

 \begin{minipage}[t]{\textwidth}
\begin{center}
\begin{tabular}{|c|c|}
\hline
  $v  $ &  {$I(v)$}
\\
\hline
$0 \pmod{12}$ & $\{(3x, \frac{v-2}{2}-2x), x=0,1,\ldots,\frac{v-4}{4}\} $\\
$6 \pmod{12}$ &  $\{(3x, \frac{v-2}{2}-2x), x=0,1,\ldots,\frac{v-2}{4}\}$\\
$3\pmod{12}$ & $\{(3x, \frac{v-1}{2}-2x), x=0,1,\ldots,\frac{v-3}{4}\}$\\
$9\pmod{12}$ & $\{(3x, \frac{v-1}{2}-2x), x=0,1,\ldots,\frac{v-1}{4}\}$\\
$6$ &  $\{(3,0)\}$\\
$12$ & $\{(3,3), (6,1)\}$\\

\hline
 \end{tabular}

\bigskip

Table 1: The set $I(v)$.

 \end{center}\end{minipage}

\vspace{4 mm}

In this paper we completely solve the spectrum problem for such
 systems; i.e., characterize the existence of uniformly resolvable
decompositions of $K_{v}$ and $K_{v}-I$ into $r$ classes of 3-paths
and $s$ classes of 3-cycles, by proving the following result:\\

\noindent \textbf{Main Theorem.} { \em
 For every integer\/ $v\geq3$, divisible by\/ $3$,
 the set\/ $URD(v;P_3,K_3)$ is identical to the set\/ $I(v)$
  given in Table 1.}

\paragraph{Notation.}

In the constructive parts of the proof we shall use the following notation,
 where $a_1,a_2,a_3$ may mean any three distinct vertices:
\begin{itemize}
 \item $(a_1,a_2,a_3)$ denotes the $3$-cycle $K_3$
  having vertex set $\{a_1,a_2,a_3\}$
and edge set $ \{\{a_1,a_2\}, \{a_2,a_3\},\{a_3,a_1\}\}$;
 \item $(a_1;a_2,a_3)$ denotes the path $P_{3}$
  having vertex set $\{a_1,a_2,a_3\}$
and edge set $ \{\{a_1,a_2\}, \{a_1,a_3\}\}$.
\end{itemize}


\section{Preliminaries and necessary conditions}

 In this section we introduce some useful definitions and give necessary conditions
 for the existence of a uniformly resolvable decomposition of $ K_v$ into
$P_3$ and $K_3$ graphs. For missing terms or results that are not
explicitly explained in the paper, the reader is referred to
\cite{CD} and its online updates. Evidently, for a uniformly
resolvable decomposition of $ K_v$ into $P_3$ and $K_3$ graphs to
exist, $v$ must be a multiple of 3.
 A (resolvable) $\cH$-decomposition
of the complete multipartite graph with $u$ parts each of size $g$
is known as a (resolvable) {\em group divisible design\/} $\cH$-(R)GDD;
the parts of size $g$ are called the groups of the design.  When
$\cH = K_n$, we call it an $n$-(R)GDD. A $3$-RGDD of type
$g^u$ exists if and only if $g(u-1)$ is even and $gu\equiv 0\pmod
3$, except when $(g,u)\in\{(2,6), (2,3), (6,3)\}$ \cite{RS}.
One can see, in particular, that a $3$-RGDD of type $2^u$ is a Nearly
Kirkman Triple System (NKTS($2u$)); we mentioned its spectrum
in the Introduction.

\begin{lemma}
\label{lemmaP1} Let \ $v\equiv 3\pmod{6}$. If\/ $(r,s)\in
URD(v;P_3,K_3)$ then\/ $(r,s)\in I(v)$.
\end{lemma}
\begin{proof}
For $v$ odd, we have $G=K_{v}$.
Let $D$ be a decomposition of  $ K_{v}$ into $r$ classes  of $P_3$
and $s$ classes of $K_3$ graphs. Counting the edges of
$K_v$ that appear in $D$ we obtain
$$
  \frac{v}{3} \cdot \left( 2r + 3s \right)
    =\frac{v(v-1)}{2} ,
$$
and hence that
\begin{equation}
  2r + 3s = \frac{3}{2} \left( v-1 \right) .
\end{equation}
This equation implies that \
$2r\equiv \frac{3}{2} (v-1)\pmod{3}$ \ and \
$3s\equiv \frac{3}{2} (v-1)\pmod{2}$. Then we obtain
\begin{itemize}
 \item
  $r\equiv0\pmod{3}$ \ and \ $s\equiv1\pmod{2}$ \ for \ $v\equiv 3\pmod{12}$,
 \item
  $r\equiv0\pmod{3}$ \ and \ $s\equiv0\pmod{2}$ \ for \ $v\equiv 9\pmod{12}$.
\end{itemize}
In either case, introducing the notation $x=r/3$,
 the equation $(1)$ determines that $s=\frac{v-1}{2}-2x$ must hold.
Since $r$ and $s$ cannot be negative, and $x$ is an integer, the value of
 $x$ has to be in the range as given in the definition of $I(v)$.
\end{proof}

\begin{lemma}
\label{lemmaP2} Let \ $v\equiv 0\pmod{6}$. If\/ $(r,s)\in
URD(v;P_3,K_3)$ then\/ $(r,s)\in I(v)$.
\end{lemma}
\begin{proof}
For $v$ even, we have $G=K_{v}-I$.
The argument is similar to the one for $v$ odd.
Let $D$ be a decomposition of  $ K_{v}-I$ into $r$ classes  of $P_3$
and $s$ classes of $3$-cycles. Counting the edges of
$K_v$ that appear in $D$ we obtain
$$
  \frac{v}{3} \cdot \left( 2r + 3s \right)
    =\frac{v(v-2)}{2} ,
$$
and hence that
\begin{equation}
  2r + 3s = \frac{3}{2} \left( v-2 \right) .
\end{equation}
This equation implies that \
$2r\equiv \frac{3}{2} (v-2)\pmod{3}$ \ and \
$3s\equiv \frac{3}{2} (v-2)\pmod{2}$. Then we obtain
\begin{itemize}
 \item
  $r\equiv0\pmod{3}$ \ and \ $s\equiv1\pmod{2}$ \ for \ $v\equiv 0\pmod{12}$,
 \item
  $r\equiv0\pmod{3}$ \ and \ $s\equiv0\pmod{2}$ \ for \ $v\equiv 6\pmod{12}$.
\end{itemize}
In either case, denoting $x=r/3$,
 the equation $(2)$ yields $s=\frac{v-1}{2}-2x$.
Since $r$ and $s$ cannot be negative, and $x$ is an integer, the value of
 $x$ has to be in the range as given in the definition of $I(v)$.
\end{proof}

\section{Small cases}

Here we handle the two exceptional cases, namely $v=6$ and $v=12$,
 for which the set $I(v)$ is slightly more restricted than for larger $v$.

\begin{lemma}
\label{lemmaD1}  $URD(6;P_3,K_3)=\{(3,0)\}$.
\end{lemma}

\begin{proof}
The case $r=0$ would correspond to an NKTS$(6)$, which does not
exist \cite{RS}.
On the other hand, for $r=3$ and $s=0$ we can take the groups to be
$\{0,1\}, \{2,3\}, \{4,5\}$ and the three classes
$\{(0;2,4), (1;3,5)\}$, $\{(2;4,1), (3;5,0)\}$,
$(4;1,3), \{(5;2,0)\}$.
\end{proof}

\begin{lemma}
\label{lemmaD2}  $URD(12;P_3,K_3)=\{(3,3),(6,1)\}$.
\end{lemma}

\begin{proof}
The case $r=0$ would correspond to an NKTS$(12)$, which does not
exist \cite{RS}.
For the other two cases, the following systems prove the assertion:
\begin{itemize}
\item $(3,3)\in URD(12;P_3,K_3)$:

$\{(1;6,a)$, $(8;0,2)$, $(3;4,9)$, $(7;5,b)\}$, \
$\{(4;7,1)$, $(5;2,b)$, $(6;8,3)$, $(9;0,a)\}$,\\
$\{(0;4,5)$, $(a;6,8)$, $(b;1,3)$, $(2;7,9)\}$; \
$\{(1,2,3)$, $(4,5,6)$, $(7,8,9)$, $(0,a,b)\}$,\\
$\{(1,5,9)$, $(4,8,b)$, $(3,7,a)$, $(2,6,0)\}$, \
$\{(1,7,0)$, $(2,4,a)$, $(3,5,8)$, $(6,9,b)\}$; \\
 $I=\{(1,8)$, $(2,b)$, $(3,0)$, $(4,9)$, $(5,a)$, $(6,7)\}$.

\item $(6,1)\in URD(12;P_3,K_3)$:

$\{(1;4,7)$, $(5;8,0)$, $(9;2,b)$, $(a;3,6)\}$, \
$\{(2;6,8)$, $(4;9,a)$, $(7;3,0)$, $(b;1,5)\}$,\\
$\{(0;4,2)$, $(3;5,9)$, $(6;7,b)$, $(8;1,a)\}$, \
$\{(1;5,6)$, $(4;8,7)$, $(9;0,a)$, $(b;3,2)\}$, \\
$\{(2;4,5)$, $(6;9,8)$, $(7;a,b)$, $(0;1,3)\}$, \
$\{(3;4,6)$, $(5;7,9)$, $(8;0,b)$, $(a;1,2)\}$;\\
$\{(1,2,3)$, $(4,5,6)$, $(7,8,9)$, $(0,a,b)\}$; \\
 $I=\{(1,9)$, $(2,7)$, $(3,8)$, $(4,b)$, $(5,a)$, $(6,0)\}$.
\end{itemize}

\vspace{-4ex}
\end{proof}

\section{Constructions for general $v$}

The key tool in this section is the following important lemma.
At the end of the paper we give some related information in the
 ``Historical remarks and acknowledgements''.

\begin{lemma}
\label{lemmaC}
 Let \ $v\equiv 0\pmod{3},$ $v\geq9$.
 The union of any two edge-disjoint
parallel classes of\/ $3$-cycles of\/ $K_v$ can be decomposed into three
 parallel classes of\/ $P_3$.
\end{lemma}
\begin{proof}
Let $Q'=\{q'_1,\dots,q'_{v/3}\}$ and $Q''=\{q''_1,\dots,q''_{v/3}\}$
 be two edge-disjoint parallel classes of $K_3$,
 whose union composes the edge set of graph $G$ on $v$ vertices.
  We represent the intersection structure of $Q'$ and $Q''$ with
   a bipartite graph $B$ with vertex bipartition $X'\cup X''$,
  where $|X'|=|X''|$ = $v/3$ and each vertex $x'_i \in X'$ and $x''_j \in X''$
  for $1 \leq  i, j \leq v/3$  corresponds to a block
   $q'_i \in Q'$ and $q''_j \in Q''$,
 respectively.
Vertex  $x'_i$ is connected to vertex $x''_j$ by an edge of $B$
  if their corresponding blocks $q'_i$ and $q''_i$ have a vertex in common.

Every block of $Q'$ ($Q''$) meets exactly three
 distinct blocks of $Q''$ ($Q'$) because each vertex appears in precisely one
  block of $Q'$ and also of $Q''$, and no vertex pair of $G$ is contained
  in blocks of both classes.
  Thus, $B$ is a $3$-regular bipartite graph.\footnote{In fact, $B$ is
  the hypergraph-theoretic dual of the 2-regular 3-uniform hypergraph whose
  hyper\-edges are the triples in $Q'\cup Q''$.}
Moreover, the edges of $B$ are in one-to-one correspondence with the vertices of $G$,
 and $G$ is the line graph of $B$.
We are going to define three edge decompositions of $B$, each of them being
 the union of $v/3$ mutually edge-disjoint copies of $P_4$ starting in $X'$ and
 ending in $X''$, in such a way that each intersecting edge-pair of $B$ occurs
 together in precisely one of those $3\times v/3$ copies of $P_4$.
Since $G$ is the line graph of $B$, this will yield the three parallel classes
  of $P_3$ as required.

It follows from the K\"onig-Hall theorem \cite{B} that the edge set
 of $B$ can be decomposed into three edge-disjoint perfect matchings;
  we view this as a proper 3-edge-coloring
  with three colors, say colors $a$, $b$, and $c$.
We define
 \begin{itemize}
  \item ${\cal P}_{abc} = \{$paths $P_4$ in $B$, starting in $X'$,
    whose color sequence is $(a,b,c)$ in this order$\}$.
 \end{itemize}
This ${\cal P}_{abc}$ is well-defined and yields an edge decomposition
 of $B$ indeed, because each color class is a perfect matching.
We define ${\cal P}_{bca}$ and ${\cal P}_{cab}$ analogously, replacing
 the sequence $(a,b,c)$ with $(b,c,a)$ and $(c,a,b)$, respectively.

It is easy to verify that the three edge decompositions
 ${\cal P}_{abc}$, ${\cal P}_{bca}$, ${\cal P}_{cab}$ of $B$
 satisfy the requirements.
For example, if an edge $e_a$ of color $a$ meets an edge $e_c$
 of color $c$ in $B$,
 then they are consecutive in one $P_4$ of ${\cal P}_{bca}$
 if $e_a\cap e_c\in X'$ or in one $P_4$ of ${\cal P}_{cab}$
 if $e_a\cap e_c\in X''$ (and they are not consecutive in any other
 $P_4$ of ${\cal P}_{abc}\cup{\cal P}_{bca}\cup{\cal P}_{cab}$).
\end{proof}

\begin{lemma}
\label{lemmaC1} For every \ $v\equiv 3\pmod{6}$, $I(v)\subseteq
URD(v;P_3,K_3)$.
\end{lemma}

\begin{proof}
Let $R_1$, $R_2$,\ldots ,$R_{\frac{v-1}{2}}$ be the parallel classes
of a resolvable KTS$(v)$. Define
 $$S_i= R_{2i+1}\cup  R_{2i+2}, \qquad i=0,1,\ldots ,\frac{v-7}{4}$$
 for \ $v\equiv 3\pmod{12}$, and
 $$T_i= R_{2i+1}\cup  R_{2i+2}, \qquad i=0,1,\ldots,\frac{v-5}{4}$$
 for \ $v\equiv 9\pmod{12}$.

By Lemma \ref{lemmaC} we know that each $S_i$ and each $T_i$ can be
 decomposed into three parallel classes of $P_3$.
Thus, in order to generate a member
  $(r,s) = (3x,\frac{v-1}{2}-2x)$ of $I(v)$,
 we apply the lemma to $(S_0,S_1,\dots,S_{x-1})$
  or to $(T_0,T_1,\dots,T_{x-1})$, depending on the
 residue of $v$ modulo 12.
The range given above for $i$ covers the entire range of $x$ in $I(v)$.
\end{proof}

\begin{lemma}
\label{lemmaC2} For every \ $v\equiv 0\pmod{6}\geq18$,
$I(v)\subseteq URD(v;P_3,K_3)$.
\end{lemma}

\begin{proof}
Start with a A $3$-RGDD of type $2^{v/3}$ \cite{RS}. This
gives that $K_v-I$ can be decomposed into $\frac{v}{3}-1$ parallel
classes of triples. Now the result can be easily obtained by using
an argument similar to the proof of Lemma \ref{lemmaC1}.
\end{proof}

\section{Conclusion}

We are now in a position to prove the main result of the paper.

\begin{thm}
For every \ $v\equiv 0\pmod{3}$,
 we have\/ $URD(v;P_3,K_3)=I(v)$.
\end{thm}
\begin{proof}
Necessity follows from Lemmas \ref{lemmaP1} and \ref{lemmaP2}.
Sufficiency follows from Lemmas \ref{lemmaD1}, \ref{lemmaD2},
\ref{lemmaC1} and \ref{lemmaC2}. This completes the proof.
\end{proof}

\paragraph{Historical remarks and acknowledgements.}

This research was done in the summer of 2012, when the second author
 visited the University of Catania.
After the presentation of
 our results at the Seventh Czech-Slovak International Symposium
 on Graph Theory, Combinatorics, Algorithms and Applications
 (Ko\v sice, July 2013),
 we learned from Alex Rosa that Lemma \ref{lemmaC} was first
 proved by Rick Wilson.
Later, Wilson informed us that he never published the lemma, but it
 was mentioned with full credit to him in a paper by John van Rees
 \cite{W}.
We thank professors Rosa and Wilson for these pieces of information.

\end{document}